\documentclass[a4paper,12pt]{article}
\usepackage[a4paper]{geometry}
\usepackage[english]{babel} 
\usepackage[T1]{fontenc} 
\usepackage[colorlinks=true,linkcolor=blue,citecolor=blue]{hyperref}%
\usepackage{amsfonts} 
\usepackage{amsthm}
\usepackage{thmtools}
\usepackage{mathtools}
\usepackage{amssymb}
\usepackage{systeme}
\usepackage{enumitem}
\usepackage{bm}
\usepackage[overload]{empheq}
\usepackage{amsmath,graphicx} 
\graphicspath{{./Images/}} 
\usepackage{blkarray}
\usepackage{longtable} 
\usepackage[bf]{caption} 
\usepackage{fancyhdr} 
\usepackage{hyperref}
\usepackage[utf8]{inputenc}
\usepackage{listings}
\usepackage{cleveref}
\usepackage{etoolbox}
\usepackage{etoolbox}
\usepackage{relsize}
\usepackage{mathtools}
\usepackage{blkarray}
\usepackage{booktabs}
\let\bbordermatrix\bordermatrix
\patchcmd{\bbordermatrix}{8.75}{4.75}{}{}
\patchcmd{\bbordermatrix}{\left(}{\left[}{}{}
\patchcmd{\bbordermatrix}{\right)}{\right]}{}{}
\theoremstyle{definition}
\newtheorem{theorem}{Theorem}[section]
\newtheorem{proposition}{Proposition}[section]

\newtheorem{lemma}{Lemma}[section]
\newtheorem{definition}{Definition}[section]

\newtheorem{corollary}{Corollary}[section]

\newcommand\norm[1]{\left\lVert#1\right\rVert}
\newcommand{\unnumberedchapter}[1]{ 
	\cleardoublepage 
	\fancyhead[RE]{{\bfseries \leftmark}} 
	\fancyhead[LO]{}}
\usepackage{emptypage} 
\usepackage{eso-pic} 

\def\sym{{\mathrm{Sym}}}
\def\C{{\mathbb{C}}}
\def\P{{\mathbb{P}}}
\usepackage{hyperref} 
\makeatletter
\def\MT@is@composite#1#2\relax{%
	\ifx\\#2\\\else
	\expandafter\def\expandafter\MT@char\expandafter{\csname\expandafter
		\string\csname\MT@encoding\endcsname
		\MT@detokenize@n{#1}-\MT@detokenize@n{#2}\endcsname}%
	\ifx\UnicodeEncodingName\@undefined\else
	\expandafter\expandafter\expandafter\MT@is@uni@comp\MT@char\iffontchar\else\fi\relax
	\fi
	\expandafter\expandafter\expandafter\MT@is@letter\MT@char\relax\relax
	\ifnum\MT@char@ < \z@
	\ifMT@xunicode
	\edef\MT@char{\MT@exp@two@c\MT@strip@prefix\meaning\MT@char>\relax}%
	\expandafter\MT@exp@two@c\expandafter\MT@is@charx\expandafter
	\MT@char\MT@charxstring\relax\relax\relax\relax\relax
	\fi
	\fi
	\fi
}
\def\MT@is@uni@comp#1\iffontchar#2\else#3\fi\relax{%
	\ifx\\#2\\\else\edef\MT@char{\iffontchar#2\fi}\fi
}
\makeatother
\usepackage{graphicx}
\usepackage{calc}
\newlength{\depthofsumsign}
\newlength{\totalheightofsumsign}
\newlength{\heightanddepthofargument}

\begin{document}
	\begin{center}
		\fontsize{15pt}{10pt}\selectfont
		\textbf{The Computation of the Euclidean Distance Degree for the Middle Catalacticant for the Binary Forms}
	\end{center}
	\vspace{0.1cm}
	\begin{center}
		\fontsize{13pt}{10pt}\selectfont
		{Belal Sefat Panah}
	\end{center}
	
\begin{abstract}
	The $n$-secant varieties to the Veronese embedding $v_{2n}(\mathbb{P}^{1})$ are hypersurfaces of degree $n+1$, denoted by  $\sigma_n(v_{2n}(\mathbb{P}^1))$. We compute the  Euclidean distance degree $\mathrm{EDdegree}$
	of $\sigma_n(v_{2n}(\mathbb{P}^1))$ for $n\le 5$ with respect to the Bombieri-Weyl quadratic form, which is maybe the most interesting case. 
	The output for $n=1,\ldots, 5$ is respectively $2, 7, 20, 53, 162$. Our main tool is the
	topological Aluffi-Harris formula. This is the first case when  the  $\mathrm{EDdegree}$ of a  $r$-secant variety to the Veronese embedding $v_{d}(\mathbb{P}^{m})$
	is computed for $r\ge 2$ and $d\ge 3$.
\end{abstract}
	 
	\section{Introduction}
		The Euclidean distance degree $\mathrm{EDdegree}$ of an algebraic variety $X$ is the number of critical points of the squared distance function from a general point,
as defined in \cite{DHOST}.
Consider $\sigma_r(v_n(\mathbb{P}^{1}))$, which are the \textit{$r$-secant} varieties of \textit{Veronese} varieties of $\mathbb{P}^1$.  The aim of this paper is to compute the $\mathrm{EDdegree}$ of $\sigma_n(v_{2n}(\mathbb{P}^1))$ for the Bombieri-Weyl forms for $n\le 5$ and provide computational evidence for $n=6, 7$.   For a general quadratic forms these numbers are known from Catanese-Trifogli formula, while for the Bombieri-Weyl quadratic form the computation is still open and it is maybe the most interesting. $\sigma_n(v_{2n}(\mathbb{P}^1))$ is a determinantal hypersurface of degree $n+1$ in $\P^{2n}$, sometimes called the middle catalecticant hypersurface. Aluffi and Harris obtain in \cite{AH} a topological formula for the $\mathrm{EDdegree}$ of a smooth complex variety , but the main obstacle to use this formula is that the $r$-secant varieties of Veronese varieties are not smooth for $r\ge 2$. In \cite{DHOST}, it was proven that the $\mathrm{EDdegree}$ of an algebraic variety is equal to the $\mathrm{EDdegree}$ of its dual. Fortunately the dual varieties of $\sigma_n(v_{2n}(\mathbb{P}^{1}))$ are smooth varieties and we can directly apply the \textit{Aluffi} and \textit{Harris} formula. Before we compute the $\mathrm{EDdegree}$ of $\sigma_n(v_{2n}(\mathbb{P}^{1}))$, we bring forward the \textit{Harmonic coordinate system} and use it as the main tool to compute the $\mathrm{EDdegree}$.
Our main results are proved in \S\ref{sec:computation}.
They can be resumed in the following
\begin{theorem}\label{thm:main}
Let $X=\sigma_n(\nu_{2n}(\mathbb{P}^n))$. Its $\mathrm{EDdegree}$ with respect to the Bombieri-Weyl quadratic form for $n\le 5$ is given by the following table
\begin{center}
	\begin{tabular}{ |p{1.5cm}|p{4cm}|}
				\hline
		$n$&$\mathrm{EDdegree}\sigma_n(\nu_{2n}(\mathbb{P}^n))$\\
		\hline
		$1$    &  2\\
		$2$    &  7\\
		$3$    &  20\\
		$4$   &  53\\
		$5$   &  162\\
				\hline
	\end{tabular}
\end{center}
\end{theorem}
The $\mathrm{EDdegree}$ of a  $r$-secant variety to the Veronese embedding $v_{d}(\mathbb{P}^{m})$ was known before only
for $r=1$ (Cartwright-Sturmfels formula, see \cite[Corollary 8.7]{DHOST}) and $d=2$ (the case of symmetric matrices, classically known by Eckart-Young Theorem). 
		\subsection{Harmonic Coordinate System}\label{sec:harmonic}	
		
	\begin{definition} Fix a nondegenerate quadratic form $q\in\sym^d\C^2$, in orthonormal coordinates we may assume $q(x,y)=x^2+y^2$. The space of homogeneous polynomials of degree $d$ in two variables $\mathrm{Sym}^{d}\mathbb{C}^2$ is not irreducible under the action of the orthogonal group $SO(2)$, but it splits in several pieces which are irreducible and each piece is given by \textit{harmonic polynomials}. The space $H_d$ of harmonic polynomials of degree $d$ is a $2$-dimensional vector space spanned by $(x-iy)^d, (x+iy)^d$. 
	In $\mathrm{Sym}^2\mathbb{C}^n$, the harmonic subspace $H_2$ corresponds to the space of traceless symmetric matrices.
The splitting in irreducible $SO(2)$-modules is 
\begin{equation}\label{eq:harmonicdec}\sym^d\C^2=\bigoplus_{i=0}^{\lfloor d/2\rfloor}q^{i}H_{d-2i}.\end{equation}
\end{definition} 
\begin{definition}[Bombieri-Weyl Form]
		Let $W$ be a space of dimension $n + 1$ , being equipped
		with a nondegenerate quadratic form $q_W$. Assume $V = \mathrm{Sym}^dW$. We choose coordinates in $W$ such that $q_W= \sum_{i=0}^{n}x_i^{2}$. There is a unique nondegenerate bilinear form $Q$ such that
		\begin{equation}\label{BW}
			Q(x^d , y^d ) = q_W(x, y)^d \ \ \ \  \forall x, y \in W,
		\end{equation}
		which is called the \textit{Bombieri-Weyl} or \textit{Frobenius} form. Since the vector space of symmetric tensors $\mathrm{Sym}^dW$ admits a basis of rank-one elements $\{x^d\mid x\in W\}$, the formula \eqref{BW} prescribes $Q(f,g)$ for every $f,g \in \mathrm{Sym}^dW$, by linearity. For the symmetric tensors $f$ and $g$ $\in \mathrm{Sym}^dW$, the Bombieri-Weyl form has the coordinate expression 
		$$Q\big(\sum_{|\alpha|=d}{d \choose \alpha} f_{\alpha} x^{\alpha},\sum_{|\alpha|=d}{d \choose \alpha} g_{\alpha}x^{\alpha}\big) = \sum_{|\alpha|=d}{d \choose \alpha}f_{\alpha}g_{\alpha}$$
		and up to a scalar factor, has the \textbf{M2} \cite{M2} implementation
		 \begin{lstlisting} 
		 	(1/d!)*diff(f,g)
	 	 \end{lstlisting} 
 	 			\end{definition}
  		
  								\begin{proposition}\label{prop:harmdecortho}
			\begin{enumerate}
			\item{i)}
			The summands in (\ref{eq:harmonicdec}) are orthogonal for the Bombieri-Weyl form.
			\item{ii)} there is a scalar $c$, depending only on $i$, $d$, such that
			$$Q(q^if,q^ig)=cQ(f, g)\qquad\forall f, g\in\sym^{d-2i}\C^2$$
			\item{iii)} in particular the forms $(x^2+y^2)^{\alpha}(x+iy)^{\beta}$ and $(x^2+y^2)^{\alpha}(x-iy)^{\beta}$ are isotropic if $\beta > 0$.
			\end{enumerate}
	\end{proposition}

		\begin{lemma}\label{2n=k+j}	
			For $0 \le j,k \le 2n$, and $n\ge1$, the bilinear product of $\langle(x+iy)^{2n-j}(x-iy)^{j},(x+iy)^{2n-k}(x-iy)^{k}\rangle$ is nonzero if and only if $2n=k+j$. 
		\end{lemma}
		\begin{proof}
			($\Rightarrow$) Suppose the product is nonzero. We split the case into the following conditions:
				\begin{enumerate}[label=(\roman*)]
					\item First,assume $j,k \le n$ therefore $2n-j\ge j, 2n-k \ge k$. The product can be rewritten as $\langle(x+iy)^{2n-2j}(x^2+y^2)^{j},(x+iy)^{2n-2k}(x^2+y^2)^{k}\rangle$, since it is nonzero and by Prop. \ref{prop:harmdecortho} i), $j=k$. Now by Prop. \ref{prop:harmdecortho} iii) the product $\langle(x+iy)^{2n-2j}(x^2+y^2)^{k},(x+iy)^{2n-2j}(x^2+y^2)^{k}\rangle$ is nonzero if $2n-2j=0$, therefore  $n=j=k$ and $2n=j+k$.
					\item Second, assume $j,k \ge n$ therefore $2n-j\le j, 2n-k \le k$. The product can be rewritten as $\langle(x^2+y^2)^{2n-j}(x-iy)^{2j-2n},(x^2+y^2)^{2n-k}(x-iy)^{2k-2n}\rangle$, since it is nonzero and by Prop. \ref{prop:harmdecortho} i), $2n-j=2n-k$, therefore $j=k$.  Now by Prop. \ref{prop:harmdecortho} iii) the product $\langle(x^2+y^2)^{2n-j}(x-iy)^{2j-2n},(x^2+y^2)^{2n-j}(x-iy)^{2j-2n}\rangle$ is nonzero if $2j-2n=0$, therefore $n=j=k$ and $2n=j+k$.
					\item Third, assume $j\le n \le k$ therefore $2n-j\ge j, 2n-k \le k$. The product can be rewritten as $\langle(x+iy)^{2n-2j}(x^2+y^2)^{j},(x^2+y^2)^{2n-k}(x-iy)^{2k-2n}\rangle$, since it is nonzero and by Prop. \ref{prop:harmdecortho} i), $j=2n-k$, therefore $2n=j+k$. By Symmetry same argument holds for $k\le n \le j$.
				\end{enumerate}
			 			($\Leftarrow$) Now suppose $2n=k+j$,  and $j \ge k$ the product can be rewritten as $\langle(x+iy)^{2n-j-k=0}(x-iy)^{j-k}(x^2+y^2)^{k},(x+iy)^{2n-2k}(x^2+y^2)^{k}\rangle$, if $j \ge k$, and by Prop. \ref{prop:harmdecortho} i), it is nonzero. Same argument holds for $j\le k$.
		\end{proof}

	We consider the last secant variety of $v_{2n}(\P^1)$ that does not fill the ambient space,
	namely $\sigma_{n}(v_{2n}(\P^1))$ which is a determinantal hypersurface of degree $n+1$ in $\P^{2n}=\P(\sym^{2n}\C^2)$ .
	\begin{lemma}\label{lem:dualmiddle} The dual variety of $\sigma_{n}(v_{2n}(\P^1))$ is the variety of squares
		$$Sq_n=\{f^2|f\in\sym^{n}\C^2\}$$
		Its projectivization in $\P^{2n}=\P(\sym^{2n}\C^2)$ is a smooth variety isomorphic to $\P^n$.
	\end{lemma}
	\begin{proof}
		Apply \cite[Corollary 3.1]{https://doi.org/10.48550/arxiv.1508.00202}, with partition $\lambda=2^n$, $\sigma_n$  is a hypersurface given by the determinant of middle Catalecticant
		map.
		Smoothness follows since $GL(\sym^{n}\C^2)$ acts transitively on $Sq_n$.
	\end{proof}
With respect to a general quadratic form we have
\begin{proposition}[Catanese-Trifogli] With respect to a general quadratic form, the EDdegree of  $\sigma_{n}(v_{2n}(\P^1))$  is $$\frac{3^{n+1}-1}{2}.$$
\end{proposition}
\begin{proof}
This follows from Catanese-Trifogli formula applied to $Sq_n$, see for example \cite[Corollary 6.1, Prop. 7.10]{DHOST}.
\end{proof}
The goal of this paper is to compute  $\mathrm{EDdegree}\sigma_n(v_{2n}(\P^1))$ with respect to the Bombieri-Weyl quadratic form. By \cite[Theorem 5.2]{DHOST} we get that $\mathrm{EDdegree}(X)=\mathrm{EDdegree}(X^\vee) $ 
	for any projective variety. The main advantage to use this formula in combination with Lemma \ref{lem:dualmiddle} is that
	the dual variety $\sigma_{n}(v_{2n}(\P^1))^\vee=Sq_{n}$ is a smooth variety.
	The isotropic quadric cuts $Sq_n$ in a divisor of $Sq_n$. Since $Sq_n$ is isomorphic to $\P^n=\P(\sym^n\C^2)$,
	we get in this way a quartic hypersurface ${\mathcal Q}$ in $\P^n$. We will see at beginning of the proof of Theorem \ref{N-osculating singularities} that this quartic is singular at least at  the two points corresponding to
	 $(x+iy)^n$ and $(x-iy)^n$. The following Theorem shows that the singular locus is indeed larger.

		\begin{theorem}\label{N-osculating singularities}
			The quartic
			$$\mathcal{Q}=\{f \in \mathrm{Sym}^{n}\mathbb{C}^2\vert \ \norm {f^2}^2  = 0, \ \text{where the norm is the Bombieri-Weyl norm.} \},$$
			\begin{enumerate}[label=(\roman*)]
				\item for $n \ge 2N +1$, $N \in \mathbb{N}^*$, contains the $N$-osculating space of rational normal curve at $(x+iy)^n$ and $(x-iy)^n$.
				\item for $n \ge 3N +1$, $N \in \mathbb{N}^*$, is singular at the $N$-osculating space of rational normal curve at $(x+iy)^n$ and $(x-iy)^n$.
			\end{enumerate} 
		\end{theorem}
		\begin{proof}
			Having Lemma \ref{2n=k+j} in mind, we choose homogeneous coordinates $(z_0,\ldots, z_n)$ in $\P^n$  defined by the following formula
			\begin{equation*}
				f=\sum_{j=0}^{n}z_j(x+iy)^{n-j}(x-iy)^j,
			\end{equation*}
			so that the two points  $(x+iy)^n$ and $(x-iy)^n$ have coordinates respectively $(1,0,\ldots, 0)$ and $(0,\ldots, 0, 1)$.
			The square $f^2$ is given by
			\begin{equation*}
				f^2=\sum_{j,k=0}^{n}z_jz_k(x+iy)^{2n-j-k}(x-iy)^{j+k},
			\end{equation*}
			and by Lemma \ref{2n=k+j} the equation of $\mathcal{Q}$ contains the monomial $z_jz_kz_pz_q$ only if $j+k+p+q=2n$, namely if it is {\it isobaric} of weight $2n$.
			In particular the two extreme monomials $z_0^4$ and $z_n^4$ do not appear in $\mathcal{Q}$ , which shows immediately that 
			the two points  $(x+iy)^n$ and $(x-iy)^n$ belong to $\mathcal{Q}$.
			The $N$-osculating space at $(x+iy)^{n}$ is spanned by $\langle(x+iy)^{n},(x+iy)^{n-1}(x-iy),(x+iy)^{n-2}(x-iy)^2, \dots ,(x+iy)^{n-N}(x-iy)^N\rangle$. For proving the statement we should verify that 
			$$V(z_0,\dots,z_N,0,\dots,0)\subseteq V(\langle f^2,f^2\rangle),$$
			which means all partial derivatives vanish at the given point $(z_0,\dots,z_N,0,\dots,0)$, which corresponds to a general point on $N$-osculating space, vanish. We verify that the terms   $z_{i_1}z_{i_2}z_{i_2}z_{i_4}$ for $i_1, i_2,i_3,i_4 \le N$ will not appear in the expression of the equation of the quartic.\\
			By the Lemma \ref{2n=k+j} this term appears when $2n-i_1-i_2=i_3+i_4$, which means $2n=\sum_{j=1}^{4} i_j \le 4N$, since each $i_j$ is bounded by $N$, and this results $n\le2N$. Therefore for $n\ge2N+1$ the quartic always contains the $N$-osculating space, and this proves (i).\\
			
			For proving (ii), we need to verify that no terms in the form $z_{i_1}^{m_1}z_{i_2}^{m_2}z_{i_2}^{m_3}z_{i_4}^{m_4}$ where $i_1, i_2,i_3 \le N$ and $i_4 \le n$, appear in the equation of the quartic. By the \textbf{Lemma.} \ref{2n=k+j} this term appears when $2n-i_1-i_2=i_3+i_4$, which means $2n=\sum_{j=1}^{3} i_j + n \le 3N+n $, and this results $n\le3N$. Therefore for $n\ge3N+1$ the quartic is singular at the $N$-osculating space of the rational normal curve at $(x+iy)^n$ and $(x-iy)^n$.
		\end{proof}
		\textbf{Conjecture.} For $n\neq3N+1$, there are no other singularities than the $N$-osculating spaces.

		\section{Computing the $\mathrm{EDdegree}$ of $\sigma_n(\nu_{2n}(\mathbb{P}^1))$}\label{sec:computation}
		In \cite{AH} \textit{Aluffi} and \textit{Harris} showed that 	the $\mathrm{EDdegree}$ of a smooth complex variety is given by the following theorem:
		\begin{theorem}\label{smooth}
			Let $X$ be a smooth subvariety of $\mathbb{P}^{n-1}$, let $Q$ be the quadric hypersurface corresponding to a nondegenerate quadratic form, assume $X  \nsubseteq Q$. Then
			\begin{equation}
				\mathrm{EDdegree}(X)={(-1)}^{\mathrm{dim}X}\chi \big(X\setminus(Q \cup H)\big),
			\end{equation} 
			where $H$ is a general hyperplane, $\mathrm{EDdegree}$ is computed with respect to $Q$  and $\chi$ is the ordinary topological Euler characteristic.
		\end{theorem}
		Then $\mathrm{EDdegree}(X)$ is given by
		\begin{equation}\label{eq: smooth complex}
			\mathrm{EDdegree}(X)={(-1)}^{\mathrm{dim}X}\big(\chi (X) - \chi(X \cap Q) - \chi (X \cap H) + \chi (X \cap Q \cap H)\big),
		\end{equation}
				Since $\sigma_n(\nu_{2n}(\mathbb{P}^1))$ are not smooth, we cannot directly apply the \textit{Aluffi} formula on $\sigma_n(\nu_{2n}(\mathbb{P}^1))$. We will compute the $\mathrm{ED}$degree of $\sigma_n(\nu_{2n}(\mathbb{P}^1))$ by its dual, since it is smooth, by Lemma \ref{lem:dualmiddle}.

		\begin{theorem}\label{thm:chiq}
						Let $Y=\sigma_n(\nu_{2n} (\mathbb{P}^1))^\vee\simeq\P^n$. Consider the quartic ${\mathcal Q}=Y\cap Q\subset \sigma_n(\nu_{2n} (\mathbb{P}^1))^\vee$, 
						defined by the Bombieri-Weyl metric
						we have $$\chi(Y\cap Q)=\left\{\begin{array}{cc}n&\textrm{\ if\ }n\textrm{\ is even}\\
						n+1&\textrm{\ if\ }n\textrm{\ is odd}\end{array}\right.$$
								\end{theorem}
		\begin{proof}
			We consider $Y\cap Q$ as a quartic hypersurface embedded in $\P^n$ with homogeneous coordinates $(z_0,\ldots, z_n)$.
		The multiplicative group $\C^*$ acts on $Y\cap Q$ by $t\cdot z_i=t^{-n+2i}z_i$ for $i=0,\ldots, n$.
		The fixed points of this action are $e_i=(0,\ldots, 1, \ldots, 0)$ for $i=0,\ldots, n$, $i\neq n/2$ when $n$ is even.
		Each fixed point $e_j$ as above defines an attractive cell $E_j=\{ x\in Q|\lim_{t\to 0} t\cdot x=e_j\}$.
		By \cite{Kon78} each $E_j$ is locally closed (this is called the Bialinicki-Birula decomposition in the smooth setting), since $Y\cap Q$ is a normal variety.
		Each $E_j$ is easily seen to be contractible.
		Hence $$ \chi(Y\cap Q)=\sum_{\footnotesize \begin{array}{c} j=0\ldots n\\j\neq n/2\end{array}}\chi(E_j)=\sum_{\footnotesize \begin{array}{c} j=0\ldots n\\j\neq n/2\end{array}} 1$$
		and the thesis follows. 
		\end{proof}	
\begin{theorem}\label{thm:chiXH}
Let $Y=\sigma_n(\nu_{2n} (\mathbb{P}^1))^\vee\simeq\P^n$, let $H$ be a general hyperplane. Then $Y\cap H$ is isomorphic to a smooth quadric of dimension $n-1$ and it satisfies
 $$\chi(Y\cap H)=\left\{\begin{array}{cc}n&\textrm{\ if\ }n\textrm{\ is even}\\
						n+1&\textrm{\ if\ }n\textrm{\ is odd}\end{array}\right.$$
								\end{theorem}
\begin{proof} The hyperplane $H$ cuts $Y\simeq\P^n$ in a quadratic hypersurface in $\P^n$.
\end{proof}
Note by Theorem \ref{thm:chiq} and Theorem \ref{thm:chiXH} it follows $$\chi(Y\cap Q)=\chi(Y\cap H).$$

\begin{corollary}\label{cor:finalED}
With respect to the Bombieri-Weyl metric, the EDdegree of $Y=\sigma_n(\nu_{2n} (\mathbb{P}^1))^\vee\simeq\P^n$
is
$$1-n+\chi(Y\cap H\cap Q) \textrm{\ if\ }n\textrm{\ is even}$$
$$1+n- \chi(Y\cap H\cap Q) \textrm{\ if\ }n\textrm{\ is odd}$$
\end{corollary}
\begin{proof} Apply (\ref{eq: smooth complex}), Theorem \ref{thm:chiq} , Theorem \ref{thm:chiXH} and the fact that $\chi(Y)=\chi(\P^n)=n+1$.
\end{proof}

		\begin{theorem}\label{EDSec2v4p1}
						Consider the \textit{Veronese embedding} of $\nu_4(\mathbb{P}^1)$. The $\mathrm{ED}$degree of $\sigma_2(\nu_4(\mathbb{P}^1))$ for the Bombieri-Weyl  quadratic form is given by $7$.
		\end{theorem}
		\begin{proof}
			 For the Bombieri-Weyl  quadratic form in \cite{OSS}, the $\mathrm{ED}$degree of $\sigma_2(\nu_4(\mathbb{P}^1))$ has been computed numerically.
			 We will prove the statement, giving a theoretical argument. Let $Y=\sigma_2(\nu_4(\mathbb{P}^1))^\vee$.
	We compute now that the intersection $Y\cap Q$ is the union of two conics tangent at two distinct points (and its Euler characteristic  is $2$).
			 Indeed, for $f \in \mathrm{Sym}^4\mathbb{C}^2$, consider the harmonic coordinate system in \S \ref{sec:harmonic}
		\begin{equation}\label{f}
			f=\sum_{j=0}^{2}z_j(x+iy)^{2-j}(x-iy)^j=z_0(x+iy)^2+z_1(x^2+y^2)+z_2(x-iy)^2,
		\end{equation}
		the intersection of $Y$ and $Q$ is a quartic which is given by 
		$$\mathcal{Q}=\{f\vert \ \norm {f^2}^2  = 0, \ \text{where the norm is Bombieri-Weyl norm.} \}$$
		By the formula of $f$ in (\refeq{f}), $f^2$ is given by
		\begin{equation}\label{f^2}
			\begin{split}
				f^2&=z_0^{2}(x+iy)^4+z_2^{2}(x-iy)^4\\
				&+ (x^2+y^2)[2z_0z_1(x+iy)^2+2z_1z_2(x-iy)^2]\\
				&+ (x^2+y^2)^2[z_1^{2}+2z_0z_2].
			\end{split}
				\end{equation}
					Each lines of the equation (\ref{f^2}) are orthogonal, and by simplifying and applying the Prop. \ref{prop:harmdecortho}  on $\norm {f^2}^2$ the quartic up to a scalar is given by
			\begin{equation}\label{Qn=2}
			\begin{split}
				\mathcal{Q}&=\norm {f^2}^2= (2^9)z_{0}^{2}z_{2}^{2}+ (2^9)z_0z_{1}^2z_2 + (2^6)z_1^{4}+(2^9)z_0z_1^{2}z_2+(2^9)z_0^{2}z_2^{2}\\
				&=1024z_0^{2}z_2^{2}+64z_1^{4}+1024z_0z_{1}^2z_2=64(16z_0^2z_2^2+z_1^4+16z_0z_1^2z_2) \\
				&=64\big((z_1^{2}+8z_0z_2)+\sqrt{48}z_0z_2\big)\big((z_1^{2}+8z_0z_2)-\sqrt{48}z_0z_2\big).
			\end{split}
		\end{equation}
				the polynomial in $z_1^2, z_0z_2$ is consequences of SO(2)-invariant which is singular at two points $(1,0,0),(0,0,1)$. The two conics in the last line of \eqref{Qn=2} are tangent at these two points.
			Since $H$ does not cut the singular points of the quartic, the intersection $(Y \cap H \cap Q)$ is given by eight distinct points.
			By applying Corollary \ref{cor:finalED} we get
					\begin{equation*}
			\mathrm{EDdegree}(\sigma_2(\nu_4(\mathbb{P}^1))=\mathrm{EDdegree}(\sigma_2(\nu_4(\mathbb{P}^1))^\vee=1-2+8=7.
		\end{equation*}
	\end{proof}
\begin{theorem}\label{EDSec3v6p1}
	Consider the \textit{Veronese embedding} of $\nu_6(\mathbb{P}^1)$, The $\mathrm{ED}$degree of $\sigma_3(\nu_6(\mathbb{P}^1))$ for the Bombieri-Weyl quadratic form is given by $20$.
\end{theorem}
\begin{proof}

	Let $Y=\sigma_3(\nu_6(\mathbb{P}^1))^\vee$. Consider the harmonic coordinate system and apply the same argument for Theorem \ref{EDSec3v6p1} for $f \in \mathrm{Sym}^{6}\mathbb{C}^2$ 
	$$	f=\sum_{j=0}^{3}z_j(x+iy)^{3-j}(x-iy)^j=z_0(x+iy)^3+z_3(x-iy)^3+(x^2+y^2)[z_1(x+iy)+z_2(x-iy)],$$
	and $f^2$ is given by 
		\begin{equation}\label{f^2n=3}
		\begin{split}
			&f^2=z_0^{2}(x+iy)^6+z_3^{2}(x-iy)^6\\
			&+ (x^2+y^2)[2z_0z_1(x+iy)^4+2z_2z_3(x-iy)^4]\\
			&+ (x^2+y^2)^2[z_1^{2}(x+iy)^2+z_2^{2}(x-iy)^2+2z_0z_2(x+iy)^2+2z_1z_3(x-iy)^2]\\
			&+(x^2+y^2)^3[2z_1z_2+2z_0z_3].\\
		\end{split}
	\end{equation}
Each lines of the equation (\ref{f^2n=3}) are orthogonal, and by computing and applying the Prop. \ref{prop:harmdecortho}  on $\norm {f^2}^2$ by \textbf{M2}, the quartic up to a scalar is given by as follow
	\begin{equation}\label{Qn=3}
		\mathcal{Q}=5z_1^2z_2^2+4z_0z_2^3+4z_1^3z_3+34z_0z_1z_2z_3+33z_0^2z_3^2.
	\end{equation}
	Recalling Theorem \ref{N-osculating singularities}, the quartic is singular at  $(1,0,0,0),(0,0,0,1)$. Moreover these points are the only singularity of the quartic, since the following system of partial equation of quartic has only solution in these two points (solved by \textbf{M2} \cite{M2})
	$$\begin{cases}
		33z_0z_3^2+17z_1z_2z_3+2z_2^3&=0\\
		5z_1z_2^2+6z_1^2z_3+17z_0z_2z_3&=0\\
		5z_1^2z_2+6z_0z_2^2+17z_0z_1z_32&=0\\
		2z_1^3+17z_0z_1z_2+33z_0^2z_3&=0.
	\end{cases}$$
		Since $H$ does not cut the singular points of the quartic, the intersection $(Y \cap H \cap Q)$
		  is a curve of type $(a,b)=(4,4)$ of degree  $(4+4)=8$  on a non-singular quadric surface in $\mathbb{P}^3$ of genus $ab -a -b +1=9$ (\cite{hartshorne1977algebraic} exercise 5.6(c)). Therefore 
	$$\chi(Y\cap H \cap Q)=2-2g=2-18=-16.$$
	Subsequently by Corollary \ref{cor:finalED}
	
	\begin{equation*}
		\mathrm{EDdegree}(\sigma_3(\nu_6(\mathbb{P}^1))=\mathrm{EDdegree}(\sigma_3(\nu_6(\mathbb{P}^1))^\vee=1+4-(-16)=20.
	\end{equation*}
\end{proof}

	\begin{theorem}\label{EDSec4v8p1}
		Consider the \textit{Veronese embedding} of $\nu_8(\mathbb{P}^1)$, The $\mathrm{ED}$degree of $\sigma_4(\nu_8(\mathbb{P}^1))$ for the Bombieri-Weyl quadratic form is  $53$.
	\end{theorem}
	\begin{proof}
	Let $Y=\sigma_4(\nu_8(\mathbb{P}^1))^\vee$. Consider the following harmonic coordinate system for $f \in \mathrm{Sym}^8(\mathbb{P}^1),$ 
		\begin{equation*}
		\begin{split}
			f&=\sum_{j=0}^{4}z_j(x+iy)^{4-j}(x-iy)^j\\
			&=z_0(x+iy)^4+z_4(x-iy)^4+(z_1(x+iy)^2+z_3(x-iy)^2)(x^2+y^2)\\
			&+z_2(x^2+y^2)^2,
		\end{split}
	\end{equation*}
	therefore $f^2$ is given by
	\begin{equation*}\label{f^2}
		\begin{split}
			f^2&=z_0^{2}(x+iy)^8+z_4^{2}(x-iy)^8\\
			&+ (x^2+y^2)[2z_0z_1(x+iy)^6+2z_4z_3(x-iy)^6]\\
			&+(x^2+y^2)^2[z_1^{2}(x+iy)^4+z_3^{2}(x-iy)^4+2z_0z_3(x+iy)^2+2z_0z_2(x+iy)^4\\
			&\ \ \ \ \ \ \ \ \ \ \ \ \ \ \ \ +2z_1z_4(x-iy)^2+2z_2z_4(x-iy)^4]\\
			&+(x^2+y^2)^3[2z_1z_2(x+iy)^2+2z_2z_3(x-iy)^2]\\
			&+(x^2+y^2)^4[z_2^2+2z_1z_3+2z_0z_4]
		\end{split}
	\end{equation*}
	Each lines of the equation is orthogonal, and the equation of the quartic is given by
	$$\mathcal{Q}=\{f\vert \ \norm {f^2}^2  = 0, \ \text{where the norm is Bombieri-Weyl norm.} \}.$$
	computing the equation by \textbf{M2} up to a scalar leads to
	\begin{equation}\label{Qd8}
		\begin{split}
			&\mathcal{Q}=\norm f^2=\\
			&z_2^4+14z_1z_2^2z_3+9z_1^2z_3^2+20z_0z_2z_3^2+20z_1^2z_2z_4+24z_0z_2^2z_4\\&+88z_0z_1z_3z_4+144z_0^2z_4^2.
		\end{split}
			\end{equation}
	The quartic is singular at tangent space at the rational normal curve at the points $(x+iy)^4$ and $(x-iy)^4$. In this case other singularities appear:
	\begin{proposition}
		The quartic \eqref{Qd8} is singular in two tangent lines at $(x+iy)^4$ and $(x-iy)^4$ and a smooth conic. Moreover the equation of the conic is given by
		\begin{equation}
			\mathcal{Q} \cap \langle (x+iy)^4,(x-iy)^4,(x^2+y^2)^2\rangle= \mathcal{Q} \cap \{z_1=z_3=0\}
			.
		\end{equation}
	\end{proposition}
	\begin{proof} Decomposing the singularity of the quartic by \textbf{M2} gives the following ideals:
		$$\{\text{\textcolor{blue}{ideal}}{(z_4,z_3,z_2)},\  \text{\textcolor{blue}{ideal}}(z_3,z_1,z_2^2+12z_0z_4),\ \text{\textcolor{blue}{ideal}}(z_2,z_1,z_0)\},$$
		which represents two tangent line at $(1,0,0,0)$ and $(0,0,0,1)$ and a conic. The equation of the conic is given by the following system of equation
	$$\begin{cases}
		z_1=0&\\
		z_3=0&\\
		z_2^2+12z_0z_4=0.
	\end{cases}$$
	\end{proof}
	Since $H$ cut the quartic in eight singular points, the intersection $(Y \cap H \cap Q)$ has $\chi$ obtained by subtracting $8$ from the $\chi$ of smooth complete intersection $(4,2)$
	in $\P^4$.
	 The complete intersection $(4,2)$ has $\chi$ equal to $64$ , see \cite{dimca2012singularities}.
	 Hence $\chi(Y \cap H \cap Q)=64-8=56$. 
	 Subsequently by Corollary \ref{cor:finalED}
	
	\begin{equation}
	\mathrm{EDdegree}(\sigma_4(\nu_8(\mathbb{P}^1))=\mathrm{EDdegree}(\sigma_4(\nu_8(\mathbb{P}^1))^\vee=1-4+56=53.
\end{equation}
\end{proof}
\begin{theorem}\label{EDSec5v10p1}
	Consider the \textit{Veronese embedding} of $\nu_{10}(\mathbb{P}^1)$, The $\mathrm{ED}$degree of $\sigma_5(\nu_{10}(\mathbb{P}^1))$ for the Bombieri-Weyl quadratic form is $162$.
\end{theorem}
\begin{proof}
Let $Y=\sigma_2(\nu_5(\mathbb{P}^{10}))^\vee$. As in the proof of Theorem \ref{EDSec3v6p1} and Theorem \ref{EDSec4v8p1},
we consider the harmonic coordinate system $(z_0,\ldots, z_5)$,
one can prove that the only singularities of $\mathcal{Q}$ are the two lines
$\{z_0=z_1=0\}$ and $\{ z_4=z_5=0\}$. 
The smooth complete intersection $(4,2)$ in $\P^5$ is a Calabi-Yau 3fold of degree $8$, whose $\chi$ is given by the coefficient of $t^3$ in the Taylor development
$$8\frac{(1+t)^6}{(1+2t)(1+4t)}$$
Such coefficient is $-176$.
Our $Y\cap H\cap Q$ has four singular points, obtained cutting the two previous lines with $H$ which is a quadric hypersurface. It can be proved (e.g. computationally) that the Milnor number of each of these points is $5$.
It follows by \cite{dimca2012singularities} that $\chi(Y\cap H\cap Q)=-176+4\cdot 5=-156$.
We get from Corollary \ref{cor:finalED}
\begin{equation}
	\mathrm{EDdegree}(\sigma_5(\nu_{10}(\mathbb{P}^1))=\mathrm{EDdegree}(\sigma_5(\nu_{10}(\mathbb{P}^1))^\vee=1+5-(-156)=162.
\end{equation}
\end{proof}
The numerical result by using \textbf{Macaulay2} for the Bombieri-Weyl form for $n=1,\dots,7$, computing the minimum distance from a random point
wth a parametrization of $Sq_n$,  are listed in the following table,
where we apply Corollary \ref{cor:finalED}:

\begin{center}
	\begin{tabular}{ |p{1.5cm}|p{2cm}|p{4cm}|}
		\hline
		$X=\mathbb{P}^n$& $\chi(\mathcal{Q}\cap H)$&$\mathrm{EDdegree}\sigma_n(\nu_{2n}(\mathbb{P})^n)$\\
		\hline
		$n=1$     &  0&  2\\
		$n=2$      &   8&  7\\
		$n=3$     & -16&  20\\
		$n=4$   &    56&  53\\
		$n=5$   &   -156&  162\\
		$n=6$   &  468&  463\\
		$n=7$   &   -1304&  1312\\
		\hline
	\end{tabular}
\end{center}
\section*{Acknowledgement}
I would like to thank my supervisor Professor Rashid Zaare Nahandi for encouraging me to do this project. A special acknowledgment goes to Professor Giorgio Ottaviani, my advisor, for his support and invaluable guidance throughout the research process.  This article is part of my doctoral thesis. 
\addtocontents{toc}{\vspace{2em}} 
\nocite{CaTri}
\nocite{zak1993tangents}
\nocite{MRW}
\nocite{DLOT}
\nocite{FO}
\nocite{CS}
\nocite{L}
\nocite{BO}
\nocite{2014}
\bibliographystyle{alpha} 
\bibliography{Paper_bibliography}
\end{document}